\documentclass[10pt, oneside]{amsart}   	% use "amsart" instead of "article" for AMSLaTeX format
\usepackage{geometry}                		% See geometry.pdf to learn the layout options. There are lots.
\usepackage{graphicx}				% Use pdf, png, jpg, or eps� with pdflatex; use eps in DVI mode
\usepackage{amssymb}

\usepackage{todonotes}

\usepackage{amssymb}
\usepackage{amsthm}
\usepackage{graphicx}
\usepackage{mathtools}
\usepackage{amsmath,amsfonts}
\usepackage{epsfig}
\usepackage{comment}
\usepackage{color}
\usepackage{amsmath}

\newcommand{\R}{{\mathbb R}}

\newcommand{\N}{{\mathbb N}}

\newcommand{\mesh}{{\rm Mesh}}
\newcommand{\len}{{\rm Len}}

\newcommand{\defeq}{:=}

\DeclareMathOperator{\diam}{diam}

\DeclareMathOperator{\lip}{lip}
\DeclareMathOperator{\LIP}{LIP}

\def\vint_#1{\mathchoice
          {\mathop{\vrule width 6pt height 3 pt depth -2.5pt
                  \kern -8pt \intop}\nolimits_{#1}}%
          {\mathop{\vrule width 5pt height 3 pt depth -2.6pt
                  \kern -6pt \intop}\nolimits_{#1}}%
          {\mathop{\vrule width 5pt height 3 pt depth -2.6pt
                  \kern -6pt \intop}\nolimits_{#1}}%
          {\mathop{\vrule width 5pt height 3 pt depth -2.6pt
                  \kern -6pt \intop}\nolimits_{#1}}}

\theoremstyle{plain}
\newtheorem{theorem}{Theorem}[section]
\newtheorem{corollary}[theorem]{Corollary}
\newtheorem{lemma}[theorem]{Lemma}

\theoremstyle{definition}

\newtheorem{remark}[theorem]{Remark}

\title{Density of Lipschitz functions in Energy}
\author{Sylvester Eriksson-Bique}
\address{Sylvester Eriksson-Bique\\
Department of Mathematics\\
Jyv\"askyl\"a University,
 P.O. Box 35 (MaD),
FI-40014, Jyv\"askyl\"a, Finland}
\email{\tt syerikss@jyu.fi}


\begin{document}

\begin{abstract}
In this paper, we show that the density in energy of Lipschitz functions in a Sobolev space $N^{1,p}(X)$ holds for all $p\in [1,\infty)$ whenever the space $X$ is complete and separable and the measure is Radon and finite on balls.  Emphatically, $p=1$ is allowed.  We also give a few   corollaries and pose questions for future work.

The proof is direct and does not involve the usual flow techniques from prior work. It also yields a new approximation technique, which has not appeared in prior work. Notable with all of this is that we do not use any form of Poincar\'e inequality or doubling assumption. The techniques are flexible and suggest a unification of a variety of existing literature on the topic.

\end{abstract}
\maketitle

\section{Introduction}

In this paper,  we study the density of Lipschitz functions in Sobolev spaces when $X$ is complete and separable, and $\mu$ is any Radon measure on $X$ which is is positive and finite on balls.  We consider the so-called Newton-Sobolev space $N^{1,p}(X)$ defined in \cite{sha00},  see also \cite{shabook}, which for $p>1$ coincides with the one introduced earlier in \cite{che99}. A function $f$ is in $N^{1,p}(X)$ if $f\in L^p(X)$ and if it has an upper gradient $g\in L^p(X)$; see equation \eqref{eq:ug}. Associated to each $f$ there is a \emph{minimal $p$-weak upper gradient} $g_f \in L^p(X)$, which plays the role of the norm of a gradient. We will give precise definitions of these and the following notations in Section \ref{sec:proof}.

Our main result proves \emph{density in energy}  or, rather, produces a sequence of Lipschitz functions which converges in energy. A sequence of functions $(f_i)_{i\in \N}$, with $f_i\in N^{1,p}(X)$ converges to $f\in N^{1,p}(X)$ in energy, if the functions $f_i$ converge to $f$ in $L^p(X)$ and if their minimal $p$-weak upper gradients $g_{f_i}$ converge to $g_f$ in $L^p(X)$. Our sequences of functions $f_i$ will be Lipschitz functions with bounded support, that is $f_i \in \LIP_b(X) \subset N^{1,p}(X)$. Our argument in fact shows more than the convergence of the minimal $p$-weak upper gradients. It shows that the asymptotic Lipschitz constants converge; 
\[
\lip_a[f](x)\defeq \lim_{r\to 0} \sup_{a\neq b\in B(x,r)} \frac{|f(a)-f(b)|}{d(a,b)},
\]
and $\lip_a[f](x) \defeq 0$, if $x$ is an isolated point. We remark, that from the definition it follows that $\lip_a[f]\geq g_f$ for any $f\in \LIP_b(X)\cap N^{1,p}(X)$, and thus, convergence of $\lip_a[f]$ is a stronger statement. Further, we give a new and flexible way to obtain these approximating functions $f_i$ without the need for tools beyond basic real analysis. 

\begin{theorem}\label{thm:densinenergy} Let $X$ be a complete and separable metric space, $p\in [1,\infty)$ and $\mu$ a Radon measure  
which is positive and finite on balls. If $f\in N^{1,p}(X)$, then there exists a sequence $f_i\in N^{1,p}(X)\cap\LIP_b(X)$ so that the following properties hold.
\begin{enumerate}
\item The functions $f_i$ converge in $L^p(X)$ to $f$, that is \[\lim_{i\to\infty} \int |f_i-f|^p d\mu=0.\]
\item The asymptotic Lipschitz constants $\lip_a[f_i]$ and the minimal $p$-weak upper gradients $g_{f_i}$ of $f_i$ converge to the minimal $p$-weak upper gradient $g_f$ of $f$ in $L^p(X)$, that is 
\[\lim_{i\to\infty} \int |g_{f_i}-g_f|^p d\mu = \lim_{i\to\infty} \int |\lip_a[f_i]-g_f|^p d\mu = 0.\] 
\end{enumerate}
\end{theorem}

The $p>1$ case of the theorem is covered by \cite{ambgigsav}, however by inexplicit and different methods. Further, the present result applies to the case $p=1$,  and gives a conceptually new way of obtaining their result. The case of $p=1$ is of particular importance in applications that use the co-area inequality, as can be seen from the concurrent work \cite{pietroeb}.

\begin{remark}
Convergence in energy is \textit{weaker} than convergence in norm. In the latter, for any $\epsilon>0$ we seek functions $f_\epsilon$ which are Lipschitz so that $\|f-f_\epsilon\|_{N^{1,p}}\leq \epsilon$. In particular, this requires that the  minimal $p$-weak upper gradients $g_{f-f_{\epsilon}}$ of the differences   $f-f_{\epsilon}$ converge to $0$ in $L^p(X)$ with $\epsilon \to 0$. However, if $f_\epsilon$ converges in energy to $f$ as $\epsilon\to 0$, then we only know that the difference  $g_f-g_{f_\epsilon}$ converges to $0$ in $L^p(X)$. Crucially, (a.e.) we have $g_{f-f_{\epsilon}} \geq |g_f-g_{f_\epsilon}|$. This means, that convergence in norm is a stronger statement than convergence in energy.

Next, consider $[0,1]^2$ equipped with the Lebesgue measure and $\ell_1$-metric $d((x_1,y_1),(x_2,y_2))=|x_1-y_1|+|x_2-y_2|$. Let $f(x,y)=x$, and $f_n(x,y)=x+n^{-1}\sin(ny)$, and consider any $p\in [1,\infty)$. Since $f_n \in N^{1,p}(X)$ is smooth, its minimal $p$-weak upper gradient is simply the norm of its gradient vector measured using the dual norm of $\ell_1$. That is, $g_{f_n}=\max\{|\partial_x f_n|, |\partial_y f_n|\} = 1 = g_f$. In particular, $f_n \to f$ in $L^p([0,1]^2)$ and $g_{f_n}=g_f=1$. However, $g_{f_n-f}=|\cos(ny)|$, which does not converge to $0$ in $L^p(X)$. Thus, the functions $f_n$ converge to $f$ in energy, but not in norm. If one uses any other uniformly convex norm on $\R^2$ to define the metric, then this phenomenon does not occur. Further, in many settings, one can take convex combinations of $f_n$ which converge in norm -- even if $f_n$ only converge in energy. Indeed, density in energy implies density in norm, if one has some finite dimensionality. This idea leads to some recent applications of the present paper, e.g. \cite[Theorem 1.9]{teriseb}.

%An example of the difference is not particularly easy to illustrate, since in any metrically doubling setting, convergence in energy implies converge in norm. This follows from, e.g. from arguments in \cite[Lemma 4.7. and Remark 4.8]{teriseb}. Let $X=\prod_{i=1}^\infty [0,i^{-1}]$, and equip $X$ with the $\ell^\infty$ metric and infinite product measure $\mu = \prod_{i=1}^\infty i\lambda|_{[0,i^{-1}]}$, where $\lambda$ is the Lebesgue measure. 
\end{remark}
Despite the weakness of density in energy, for many arguments it suffices. One application is for proving equivalence of various types of Poincar\'e inequalities. We say that a pair $(u,g)$ satisfies a $p$-Poincar\'e inequality (with constants $(C,\Lambda)$) if for each ball $B(x,r) \subset X$
\begin{equation}\label{p-PI}
\vint_{B(x,r)} |u-u_{B(x,r)}| ~d\mu \leq Cr \left( \vint_{B(x,\Lambda r)} g^p ~d\mu \right)^{1/p},
\end{equation}
where we define the average by $f_A\defeq \vint_A f ~d\mu \defeq \frac{1}{\mu(A)} \int_A f ~d\mu$, when the final expression is well defined. For the statement of the following corollary, we define:
\[\lip[f](x)\defeq \liminf_{y\to x, x\neq y} \frac{|f(x)-f(y)|}{d(x,y)}\]
(and $\lip[f](x)\defeq 0$, if $x$ is isolated).

\begin{corollary} Let $X$ be a complete and separable metric space, $p\in [1,\infty)$ and $\mu$ a Radon measure  
which is positive and finite on balls. For any fixed constants $(C,\Lambda) \in (0,\infty)^2$ the following three are equivalent.

\begin{enumerate}
\item For every Lipschitz function $f:X \to \R$ the pair $(f,\lip_a[f])$ satisfies a $p$-Poincar\'e inequality with constants $(C,\Lambda)$.
\item For every Lipschitz function $f:X \to \R$  the pair $(f,\lip[f])$ satisfies a $p$-Poincar\'e inequality with constants $(C,\Lambda)$.
\item For every  $f \in N^{1,p}(X)$ the pair $(f,g_f)$ satisfies a $p$-Poincar\'e inequality with constants $(C,\Lambda)$.
\end{enumerate}
\end{corollary}
The statement does not depend on any doubling or properness assumption, as was for example assumed in \cite{keith03}. The proof is a simple exercise left to the reader of using Theorem \ref{thm:densinenergy} and the fact that for Lipschitz functions $g_f \leq \lip[f] \leq \lip_a[f]$ (almost everywhere).  We note that the assumption of completeness can not be removed, as seen from \cite{keithzhong, koskelaremov}. 
%The fact that $g_{f_i}\leq g$ ensures that $g_{f_i}$ are equi-integrable, which plays a crucial role in our presentation. Indeed, the reflexivity of $L^p$, which is traditionally used for weak compactness, can now be replaced by Dunford-Pettis in the $p=1$ case. 

A further application of a technical nature is the following, which was pointed out to us by Elefterios Soultanis.

\begin{corollary} Assume that $X$ is complete and separable and equipped with a Radon measure which is positive and finite on balls.  If $f \in N^{1,p}(X)$, then there is a Borel function $\tilde{f}\in N^{1,p}(X)$ so that $\tilde{f}=f$ almost everywhere.
\end{corollary}

Note that, we only require that $f$ is measurable. Further, note that the Newton-Sobolev condition involves a pointwise consideration. Thus,  a direct modification using Borel regularity does not yield the result, since it may break the Newton condition. The proof of this corollary follows immediately from Theorem \ref{thm:densinenergy} by considering a subsequence of $f_i$ converging pointwise and their limit together with \cite[Lemma 7.3]{haj03}. We remark, that \emph{a posteriori} also $f=\tilde{f}$ at capacity almost every point, or quasi-everywhere, see \cite[Corollary 3.3]{sha00}. 

The question of the density of continuous and Lipschitz functions is also crucial in other contexts. For example, quasi-continuity properties of Sobolev functions are implied by density in norm,  see \cite{bjornnages} and \cite{malylukas}. While we only get density in energy, it seems our techniques could have something to say in these contexts as well.  In conclusion, it appears that density of continuous and Lipschitz functions in energy in Newton-type Sobolev spaces defined using upper gradients is more generic than appears from existing literature.

\subsection{Approximation scheme}
Before we describe our scheme, which may seem a little confusing, we wish to show how it arises naturally from prior work and to survey existing approximation schemes.  In the following discussion, $f:X\to \R$ will be a Sobolev function, and $\tilde{f}$ will denote its approximation.

In Euclidean spaces, and Lie groups, the simplest way to approximate an $L^p(X)$-function $f$ is via convolution with $\tilde{f}=f\ast \phi_n$, where $\phi_n$ is some approximation of unity. Even without a group structure, one can mimic this process on manifolds, and even some CAT(0) spaces, using various center of mass constructions. See e.g. \cite{karchergrove, karchermollification}. 

In general metric spaces, such a convolution method is missing. One rather old method, is to employ so called discrete convolutions:
\[\tilde{f}(x)= \sum_n \psi_n f_{B_n},\]
where $B_n$ are some balls which cover the space (with finite overlap), and $\psi_n$ is a partition of unity subordinate to $B_n$. Such an approximation goes back to Coifman and Weiss \cite{coifmanweiss}. When $X$ is measure doubling (for some $D\in [1,\infty)$ we have $\mu(B(x,2r))\leq D\mu(B(x,r))$ for each $x\in X, r>0$) and satisfies a Poincar\'e inequality, one can perform a discrete convolution in such a way that $\tilde{f}$ approximates $f$ in the $L^p(X)$-norm, and so that the $N^{1,p}(X)$-norm of $\tilde{f}$ is controlled. For such results, see \cite{kinnunenlatvala}. Further, for applications in other contexts, and earlier results see \cite[P. 290–292]{semmes} and \cite{macias}.

Discrete convolutions have been quite successful in analysis on metric spaces. In fact, the papers using this technique are too numerous to list here. But, some highlights we found are \cite{lahtishan, kinnunenlatvala, kinnunenkortenages}. These papers require strong assumptions: a doubling measure and a Poincar\'e inequality. It is unclear to us, if a discrete convolution approach could be extended to more general settings. Another issue with discrete convolutions, is that even with the aforementioned strong assumptions, they fail to prove Theorem \ref{thm:densinenergy}. Indeed, the approximating functions $\tilde{f}$ have minimal $p$-weak upper gradients $g_{\tilde{f}}$ with a bound on their $L^p(X)$-norms -- but convergence of the minimal $p$-weak upper gradients may fail.

In order to get convergence of gradients, and to weaken assumptions, we need to look beyond. The seminal paper of Cheeger \cite{che99} includes \emph{several} useful approximation schemes, which apply to spaces with a doubling measure and a Poincar\'e inequality:
\begin{enumerate}
\item Theorem 4.24: a sub-level set approach employing an approximation based on the McShane extension in the proof of the existence of a differential for Sobolev functions,
\item Lemma 5.2: an extension using curves similar to the Equation \eqref{eq:approx} (see below), which was used to prove a version of Theorem \ref{thm:densinenergy} and conclude $g_f=\lip[f]$ for length spaces; and 
\item Theorem 6.5: an approximation with piecewise distance function, which was used to prove that $g_f=\lip[f]$ without the length-space assumption.
\end{enumerate}
These methods get much closer to the approximation method introduced in the present paper -- even though the proofs written in \cite{che99} involved the strong assumptions of a Poincar\'e inequality and doubling. These arguments thus do not directly answer our questions. However, with these stronger assumptions,  Theorem 6.5 does prove our Theorem \ref{thm:densinenergy} for $p>1$. 

Each of Cheeger's methods involved \emph{extending} $f|_A$ out from a subset $A\subset X$ -- or approximately interpolating the function while attempting to ensure that $\sup_{a\in A}|\tilde{f}(a)-f(a)|$ is as small as possible. This idea, that approximation and extension problems are connected, is the central theme of our paper. We briefly describe the first two approximations of Cheeger in slightly more detail to illustrate the connection to our method.
 
In the sub-level set approach, one obtains $\tilde{f}$ by taking a sub-level set $A \subset X$ of the Hardy-Littlewood maximal function $M(g_f^p)$, and employing a McShane Lipschitz extension. Here, Cheeger used the fact that $f|_A$ is Lipschitz. This follows from a pointwise Poincar\'e inequality.
This approximation has the remarkable Lusin property: the approximating function $\tilde{f}$ actually agrees with the function $f$ on large measure subsets.   Together with locality of the minimal $p$-weak upper gradient, one gets that $\tilde{f}-f$ has small $N^{1,p}(X)$-norm.
If one wishes to remove the Poincar\'e inequality assumption, then one needs substantially different tools. Further, without a Poincar\'e inequality, one must give up the (Lipschitz-)Lusin property because the function $f \in N^{1,p}(X)$ may not, in general, be Lipschitz on any positive measure subset. 

The second approximation in \cite{che99} takes the form of
\begin{equation}\label{eq:approx}
\tilde{f}(x) = \inf_{\gamma:A \leadsto y} f(\gamma(0)) + \int_\gamma g ~ds,
\end{equation}
where $A\subset X$ is some compact set of points, and the infimum is taken over rectifiable curves $\gamma:[0,1]\to X$ connecting $A$ to $x$ and the integral $\int_\gamma g ~ds$ is the usual curve integral. While the proof, as written, in \cite{che99} does use a Poincar\'e inequality, the full strength of it is not really needed. Indeed, most of the lemmas in \cite{che99} do not use this assumption, and the proof of Lemma 5.2 could be rewritten by using the Lusin property and choosing a set $A$ where $f|_A$ is continuous. By such a modification, the proof would apply to a proper length space. To our knowledge, this has not been observed before. However, we omit the details of this claim, since our main theorem contains a stronger result.

The idea with the previous approximations is to be close to $f$ on the set $A$, while insisting on $g$ being an upper gradient. Such functions have arisen in other settings, where one wishes to prescribe a given upper gradient: see e.g. \cite[Lemma 3.1]{bjornnages}. The main technical problem with the definition in Equation \eqref{eq:approx}, is that it implicitly insists on the existence of rectifiable curves $\gamma$.  Without such curves $\tilde{f}$ may even fail to be continuous -- indeed even the measurability is non-trivial as can be seen for example from the work in \cite{jarvempaa}. 

The issue of a lack of curves has already been identified, and resolved, in limited instances. When proving that a Poincar\'e inequality implies quasiconvexity, \emph{a priori} we are not allowed to assume the existence of any curves.  See for example the beautiful discussion in \cite[Proposition 4.4]{hajkos} where a version of this fact is proved -- or the proof in \cite[Theorem 17.1]{che99} which is originally due to Semmes.  The proof involves ``testing'' the Poincar\'e inequality with functions of the form:
\[
f_{\epsilon,x}(y)\defeq \inf_{p_0,\dots, p_n} \sum_{k=0}^{n-1} d(p_i,p_{i+1}),
\]
with the infimum is taken over all sequences of points, called discrete paths, $p_0,\dots, p_n$ with $d(p_i,p_{i+1})\leq \epsilon$ and $p_0=x,p_n=y$, and $x\in X$ a fixed point. Indeed, our approximating expression will resemble and generalize this expression.

Our approximation combines the two main aspects from before: discrete paths, and extending from a subset $A\subset X$. First, we explain the approximation for non-negative functions. The approximation is defined using given data: A non-negative function $f:X\to [0,M]$ to approximate which is bounded by $M>0$,  a continuous bounded non-negative function $g:X\to [0,\infty)$ which is our desired upper gradient,  a set $A$ of values where $f|_A$ is continuous, and a scale $\delta>0$ for our approximation.

The formula for the approximating function is:
\[
\tilde{f}(x) = \min\left\{\inf_{p_0,\dots,p_n} f(p_0) + \sum_{k=0}^{n-1}g(p_k)d(p_k,p_{k+1}),M\right\}
\]
where the infimum is taken over all discrete paths  $p_0,\dots,p_n$ with $p_0 \in A$, $d(p_k,p_{k+1})\leq\delta$ and $p_n=x$. There are a few elementary observations worth noting:
\begin{enumerate}
\item The function $\tilde{f}$ is automatically $\max\{\sup_{x\in X}g(x), M/\delta\}$-Lipschitz, and we have $\lip_a(\tilde{f}) \leq g$. Indeed, if $d(x,y)\leq \delta$ for some $x,y\in X$, by concatenating discrete paths, we obtain a bound for $|\tilde{f}(x)-\tilde{f}(y)|\leq \max\{g(x),g(y)\}d(x,y)$. By boundedness, if $d(x,y)>\delta$, then $|\tilde{f}(x)-\tilde{f}(y)|\leq M \leq \frac{2M}{\delta} d(x,y)$. Further details will be given later.
\item We have $\tilde{f}(x)\leq f(x)$ for each $x\in A$. In the infimum, we can choose the discrete path $P=(x,x)$. Thus, we have $\tilde{f}(x) \leq f(x)+g(x)d(x,x)=f(x)$.
\end{enumerate}
The difficulty lies then in choosing a $g$ appropriately and sending $\delta \to 0$, from whence one can show that $\tilde{f}|_A$ must converge to $f|_A$ pointwise. Here, a refined version of Arzela-Ascoli is used as part of a compactness and contradiction argument. If the convergence were to fail, then we would get a sequence of discrete paths converging to a curve, which violates the upper gradient inequality \eqref{eq:ug}. This compactness argument is clearly explained in \cite[Proposition 2.17]{heinonenkoskela} and in \cite[Lemma 5.18]{che99}. According to Cheeger, this argument has also been used by Rickman and Ziemer \cite[Remark 5.26]{che99}.

 For the technically minded, we already mention that properness, which is usually assumed,  is avoided by appropriately choosing $g$ to penalize paths that form non-compact families. Indeed, paths that travel "far" away from certain compact sets must have small "modulus", as we will later make precise. This is a new argument and seems to be useful in other settings as well.  However, with properness, our proof would be considerably simpler -- and this will be indicated in the proof.

The previous formula applies directly only to bounded and non-negative functions. For non-negative, or non-bounded functions, one first does a truncation and applies a cutoff function. Then, the previous approximation scheme is applied to the positive and negative part individually. We note that modifications of the formula yield approximations directly for any function $f\in N^{1,p}(X)$. However, the scheme here simplifies the proof slightly.

Finally, for the sake of completeness, we wish to mention another enormously successful approximation method, which employs a heat-flow on functions and shows that it satisfies desired estimates \cite{ambgigsav}. This approximation is implicit, since the method does not directly furnish the approximation, but shows that it exists. This requires defining a functional, via relaxation of a Dirichlet energy, and needs lower-semicontinuity for it in $L^2(X)$. By showing that two different expressions define the gradient of this flow, one obtains the existence of approximations for Sobolev functions. While very general, and powerful, this approach has two main shortcomings. The first is the inexplicit nature of the approximating function and lack of any pointwise control. Another problem is in the $p=1$ case for $N^{1,1}(X)$. When $p=1$, the relaxation approach does not give information on $N^{1,1}$, but instead on functions of bounded variation. To give results for this borderline case, we needed to introduce the methods of the present paper. We remark, that this case is still quite important, since it is connected via the co-area formula to (modulus) estimates on curves and surfaces in the space.

%The approximation is a natural discretization of the one by Cheeger. Our approximation scheme also gives more pointwise control on the approximating functions $f_i$. It also seems to apply to many other Sobolev-type function spaces which have been considered, see e.g. Orlicz-Sobolev spaces [], Lorenz-Sobolev  spaces and variable exponent Sobolev spaces [].

\subsection{Further questions}

The methods of this paper are likely to apply to a host of other Sobolev type spaces and lead to interesting further questions. A number of works on approximations have appeared  in different settings, see e.g. Orlicz-Sobolev spaces \cite{helisobolev},  Lorenz-Sobolev  spaces \cite{lorentzsob} and variable exponent Sobolev spaces \cite{hasto1,hasto2}.  One can even study these questions with a general Banach function space norm, see e.g. \cite{malylukas, minimallukas}. This list is far from exhaustive.  Indeed, a variety of authors have asked for necessary and sufficient conditions for the density of Lipschitz functions in these settings -- and we suggest that completeness and separability suffice,  with perhaps minimal further assumptions when an upper gradient is used. It is important to note, however, that the situation is quite different for the Sobolev space, often denoted $W^{1,p}(X)$, which is defined using a distributional gradient, see e.g.  \cite{hasto1} for such issues in a variable exponent case. 
The techniques here suggest, that the questions on density in this different setting are equivalent with $N^{1,p}(X)=W^{1,p}(X)$, which is a type of regularity statement. 

Another question is when (locally) Lipschitz functions are dense in $N^{1,p}(\Omega)$ when $\Omega$ is a domain -- i.e open and connected -- in a complete and separable space $X$. We use completeness in our arguments, and additional care is needed close to the boundary of $\Omega$.  In some cases, when $\Omega$ is say a slit disk in the plane $B(0,1) \setminus (0,1)\times\{0\} \subset \R^2$, one would not expect such a density for globally Lipschitz functions.  However, it may be that some minimal assumption would guarantee density of locally Lipschitz functions.

A final,  and seemingly difficult question, is if Lipschitz functions are always actually dense in $N^{1,p}(X)$ in norm,  and not just in energy (when $X$ is complete).  
In \cite{ambgigsav} this is shown for $p>1$ under a further assumption involving finite-dimensionality, or a covering by finite-dimensional parts. In a concurrent work with Elefterios Soultanis, we have employed techniques from this paper to get the $p=1$ case, with a similar assumption \cite{teriseb}.

\vskip.3cm
\noindent \textbf{Outline:} The proof of Theorem \ref{thm:densinenergy} will be at the end of Section \ref{sec:proof}. At the beginning of that section are three preliminary subsections, which will describe the terminology, basic properties of the approximating functions, and some useful lemmas for discrete paths.

\vskip.3cm
\noindent \textbf{Acknowledgements:} The author was supported by the Finnish Academy under Research Postdoctoral Grant No. 330048. He is also thankful to Jeff Cheeger, Elefterios Soultanis, Pietro Poggi-Corradini,  and Nageswari Shanmugalingam for many insightful conversations that inspired the results of this paper.  Shanmugalingam and Soultanis also gave helpful comments on an early draft of the paper. While the paper was a preprint,  we also received helpful comments from anonymous referees, which improved the presentation.

%We leave open the question of if Lipschitz functions are dense in the norm of $N^{1,p}$, and not just in energy, without any further assumptions. 

%To demonstrate these techniques, we solve a problem on variable exponent spaces $N^{1,p(x)}(\R^n)=\{f: f\in L^{p(x)}, |\nabla f| \in L^{p(x)}\}$, and show that Lipschitz functions are dense in this space as long as $1\leq p(x)<\infty$ everywhere, and that $N^{1,p(x)}(\R^n)$ is reflexive whenver $\inf_{x\in\R^n} p(x)=p_{-1}>1$. The related question of if $W^{1,p(x)}$ has the same properties is thus equivalent to if $W^{1,p(x)}=N^{1,p(x)}$.

%
%
%
%\begin{theorem}\label{thm:density} Let $X$ be complete, separable and $\mu$ a Radon measure finite on balls.  Suppose that that there are countably many sets $X_i$, indexed by $i\in I$, so that
%
%\[
%\mu(X\setminus \bigcup_{i\in I} X_i) = 0
%\]
%and $\dim_H(X_i) < n_i$. Then, Lipschitz functions are dense in $N^{1,p}$ for every $p\in [1,\infty)$.
%\end{theorem}

\section{Proof of Approximation}\label{sec:proof}

\subsection{Preliminaries}
Throughout this section $p \in [1,\infty)$ and $X$ is a complete and separable metric space equipped with a Radon measure finite on balls, i.e.  $0<\mu(B(x,r))<\infty$ for each $x\in X, r>0$. A rectifiable curve is a map $\gamma:I\to \R$ from a compact subset $I\subset \R$ with finite length $\len(\gamma)$.

Recall, that as introduced by Heinonen and Koskela in \cite{heinonenkoskela}, a non-negative Borel function $g:X\to[0,\infty]$ is a (true) upper gradient for $f:X\to [-\infty,\infty]$, if
\begin{equation}\label{eq:ug}
\int_\gamma g ~ds \geq |f(\gamma(1))-f(\gamma(0))|,
\end{equation}
for any rectifiable curve $\gamma:[0,1]\to X$. In the case of infinities on the right hand side, when the difference may not be well-defined, we insist that the left hand side is infinity as well. 

The property of being an upper gradient is not closed under $L^p(X)$-convergence. Thus, one introduces the notion of a \emph{$p$-weak upper gradient}. One simple definition for this is that if ${\rm UG}(f)\subset L^p(X)$ is the collection of upper gradients  in $L^p(X)$ for $f$, then $g$ is a $p$-weak upper gradient if $g\in \overline{{\rm UG}(f)}$ (i.e. lies in the closure). Equivalently, this notion can be defined using a notion of ``a.e.'' curves coming from the concept of a modulus of curve families. We say that $g$ is a $p$-weak upper gradient, if \eqref{eq:ug} holds for $p$-a.e. rectifiable curve $\gamma$: i.e. if there exists a $h\in L^p(X)$, so that for every rectifiable curve $\gamma$ for which $\int_\gamma h \, ds < \infty$ we have inequality \eqref{eq:ug}. We refer the reader to \cite{haj03,sha00} for details on modulus.

\begin{remark}\label{rmk:vitalicaratheodory}
We prefer not to define modulus of curve families here,  as we do not need it.  Instead, we  note that the only property we need is that if $g$ is a $p$-weak upper gradient for $f$,  then for any $\epsilon>0$ there is a lower semicontinuous (true) upper gradient $g_\epsilon \geq g$ with  $\int |g-g_\epsilon|^p ~d\mu \leq \epsilon.$ This can be easily seen from $g\in \overline{{\rm UG}(f)}$ together with the fact that any function in $L^p(X)$ can be approximated from above by a lower semi-continuous function (by the Vitali-Caratheodory Theorem).
For the details we refer to \cite[Sections 4.2 and Sections 5--6]{shabook}.  
\end{remark}

We define $N^{1,p}(X)$ as the collection of $f\in L^p(X)$ so that there exists a Borel upper gradient $g \in L^p(X)$. The functions in this collections are called Newton, Sobolev, or Newton-Sobolev functions. The space is called either the Newton, Sobolev, or Newton-Sobolev space. We define 
\begin{equation}\label{eq:norm}
\|f\|_{N^{1,p}}=\left(\inf_{g\in \overline{{\rm UG}(f)}} \|f\|_{L^p}^p + \|g\|_{L^p}^p\right)^{1/p},
\end{equation}
where the infimum is taken over all $p$-weak upper gradients of $f$ (or, equivalently, over all true upper gradients $g$ of $f$). By \cite[Theorem 7.16]{haj03} there always exists a minimal $p$-weak upper gradient $g_f \in \overline{{\rm UG}(f)}\subset L^p(X)$ which attains the infimum in Equation \eqref{eq:norm}, and which satisfies \eqref{eq:ug} for $p$-almost every curve. See also \cite{sha00} for the $p>1$ case. Since the set $\overline{{\rm UG}(f)}$ of $p$-weak upper gradients satisfies the lattice property \cite[Corollary 6.3.12]{shabook}, we have that the minimal $p$-weak upper gradient satisfies the additional property that if $\tilde{g}$ is another $p$-weak upper gradient for $f$, then  $\tilde{g}\geq g_f$ a.e.

We also define the Lipschitz and  asymptotic Lipschitz constant  for a function $f:X \to \R$
\begin{equation}\label{eq:alip}
\LIP[f](A) = \sup_{x,y \in A, x\neq y} \frac{|f(x)-f(y)|}{d(x,y)}, \lip_a[f](x)=\lim_{r\to 0} \LIP[f](B(x,r)).
\end{equation}
When $A$ is a singleton, we interpret $\LIP[f](A)=0$. A function $f$ is \emph{Lipschitz} if $\LIP[f](X)<\infty$, and $\LIP_b(X)$ is the collection of Lipschitz $f$, with bounded support (i.e. there exists some ball $B(x_0,R)\subset X$ so that $f(y)=0$ for each $y\not\in B(x_0,R)$). A function $f$ is called $L$-Lipschitz if $\LIP[f](X)<L$. We note that $\LIP_b(X)\subset N^{1,p}(X)$ for each $p\in [1,\infty)$ and $g_f\leq \lip_a[f]$ for each $f\in \LIP_b(X)$.

We need the following two lemmas. First we need an argument which replaces the usual application of reflexivity when $p>1$.  For the following proof, we will define that a collection $\mathcal{F}$ of functions is \emph{equi-$p$-integrable} if, for every $\delta>0$, there exists, a set $\Omega_\delta \subset X$ with $\mu(\Omega_\delta)<\infty$ and an $\eta>0$, so that for every $f\in \mathcal{F}$ the following two hold: \textbf{a)} $\int_{X \setminus \Omega_\delta} |f|^p ~d\mu \leq \delta$ and \textbf{b)} for any $E \subset X$ with $\mu(E)<\eta$ we have that $\int_E |f|^p ~d\mu \leq \delta$.

The first lemma shows that minimal $p$-weak upper gradients converge, if they possess some upper bounds which converge.

\begin{lemma}\label{lem:convergence} Suppose that $f_i,f\in N^{1,p}(X)$ are functions with minimal $p$-weak upper gradients $g_{f_i},g_f \in L^p(X)$ and that $f_i \to f$ in $L^p(X)$. Suppose further that $\tilde{g}_i \in L^p(X)$ are functions which converge $\tilde{g}_i \to g_f$ in $L^p(X)$.  If $g_{f_i}\leq \tilde{g}_i$, then $g_{f_i}\to g_f$ in $L^p(X)$. 
\end{lemma}

\begin{proof} We first show that that every subsequence of $g_{f_i}$ has a further subsequence that converges weakly to $g_f$. This shows that $g_{f_i}$ converges weakly to $g_f$. Indeed, up to reindexing, it suffices to find such a subsequence of the original $g_{f_i}$,  and throughout the proof we will not bother with reindexing. 

First, if $p>1$, then we can extract a subsequence where $g_{f_i}$ converges weakly. If $p=1$ we note that $g_{f_i}\leq \tilde{g}_i$ and since $\tilde{g}_i$ converge in norm, we have that $g_{f_i}$ are equi $1$-integrable. Therefore, by the Dunford-Pettis Theorem \cite[Theorem IV.8.9]{dunsch}\footnote{In our terminology, Dunfort-Pettis Theorem states that any equi-$1$-integrable collection of functions is weakly sequentially compact.  See for example \cite[Theorem IV.8.9]{dunsch}.}, the sequence has a weakly convergent subsequence. By passing to this subsequence assume thus that $g_{f_i}$ converges weakly to a function $\tilde{g}$. By lower semi-continuity 
\begin{equation}\label{eq:integral}
\int \tilde{g}^p d\mu \leq \liminf_{i\to \infty} \int g_{f_i}^p d\mu \leq \liminf_{i\to \infty} \int \tilde{g_{i}}^p = \int g_f^p d\mu.
\end{equation} By passing to a subsequence, we can also take $\tilde{g}_i$ to converge pointwise almost everywhere to $g$.

By Mazur's lemma, we have convex combinations of $g_{f_i}$ converging strongly to $\tilde{g}$. That is, for each $N\in \N$, we can find some $L_N\in \N$ with $L_N\geq N$ and non-negative constants $\alpha_{k,N}\in [0,1]$ for $k\in \N \cap [N,L_N]$ with the following properties. 
\begin{itemize}
\item $\sum_{k=N}^{L_N} \alpha_{k,N}=1$.
\item The functions $g_N \defeq \sum_{k=N}^{L_N} \alpha_{k,N} g_{f_k}$ satisfy $g_N \to \tilde{g}$ in $L^p(X)$.
\end{itemize}
Further, the function $\tilde{f}_N \defeq \sum_{k=N}^{L_N} \alpha_{k,N} f_k$, has $g_N$ as a  $p$-weak upper gradient and $\tilde{f}_N \to f$ in $L^p(X)$. Thus, $\tilde{g}$ is a $p$-weak upper gradient for $f$ by \cite[Proof of Lemma 3.6.]{sha00}. By the minimality of $g_f$ and the lattice property of weak upper gradients we obtain that a.e. $\tilde{g}\geq g_f$. However, combined with Equation \eqref{eq:integral} this gives $\tilde{g}=g_f$ almost everywhere, and that the sequence $g_{f_i}$ converges weakly to $g_f$.

By the equi-$p$-integrability, we only need to prove that $g_{f_i}$ converges to $g_f$ in measure: 
\[\lim_{n\to\infty}\mu(\{x \in A: |g_f(x)-g_{f_n}(x)|>\epsilon\})=0\] 
for any $A\subset X$ with $\mu(A)<\infty$. Fix such a set $A\subset X$ and $\epsilon>0$. We have that $\tilde{g}_i\to g_f$ pointwise, and so $\limsup_{i\to \infty} g_{f_i} \leq g_f$ almost everywhere. Thus, $\limsup_{n\to\infty}\mu(\{x\in A : g_{f_n}(x) - g_{f}(x) > \epsilon\})=0.$

With this and equi-integrability, we then have $\limsup_{n\to\infty}\mu(\{x : |g_f(x) - g_{f_n}(x)| > \epsilon\}) \leq \limsup_{n\to \infty} \frac{1}{\epsilon}\int_A g_f(x)-g_{f_n}(x) d\mu = 0,$ where we used weak convergence. Thus, the sequence $g_{f_n}$ converges in measure as claimed.
%since these convex combinations come from a sequence converging to $f$, we have that $\tilde{g}$ is an upper gradient for $f$. Therefore $\tilde{g} \geq g_f$ almost everywhere, since $g_f$ is minimal. Indeed, $\tilde{g}=g_f$ almost everywhere, and we have that $g_{f_i}$ converging weakly to $g_f$.

%Since we have $\tilde{g}_i \to g_f$ almost everywhere by our choices, we get $\limsup_{i\to\infty} g_{f_i}(x) \leq g_f(x)$, and thus by DU we have that $g_{f_i}$ actually converges strongly. 

\end{proof}

Next, we need a version of the Arzel\`a-Ascoli Theorem, which is easy to prove using standard techniques. We state this lemma to highlight that, whereas the usual literature on the topic employs properness of $X$ and some boundedness assumption for the sequence $\gamma_k$, we can avoid this entirely by ensuring pre-compactness of certain sets. Indeed, we will even use the Lemma with $Y=\ell_\infty(\N)$, which is not proper.

Note that a curve $\gamma:[0,1]\to X$ is $L$-Lipschitz if $d(\gamma(s),\gamma(t))\leq L|s-t|$ for every $s,t\in [0,1]$.

\begin{lemma}\label{lem:arzela-ascoli} Let $L \in [0,\infty)$.  Suppose that $Y$ is a complete metric space. Let $\gamma_k:[0,1]\to Y$ be a sequence of $L$-Lipschitz curves, so that for every $t\in [0,1]$ the set $A_t=\{\gamma_k(t) : k\in \N\}$ is pre-compact in $Y$. There exists a subsequence of $\gamma_k$ which converges uniformly to a $L$-Lipschitz curve $\gamma$.
\end{lemma}

\subsection{Approximating function}

We collect in this subsection the definition of our approximating functions and their main properties. 

A \emph{discrete path} $P$ (or simply ``a path'' $P$) is a sequence of points $P=(p_0,\dots, p_n)$ with $p_k\in X$ for each $k=0,\dots, n$, and $n\in \N$ with $n\geq 1$. We define the mesh of $P$ by $\mesh(P)\defeq \max_{k=0,\dots, n-1} d(p_k,p_{k+1})$, the diameter of $P$ by $\diam(P)\defeq \max_{k,l} d(p_k,p_{l})$ and the length of $P$ by $\len(P)\defeq \sum_{k=0}^{n-1} d(p_k,p_{k+1})$. By a slight abuse of notation, we will write $p\in P$ if there is a $k=0,\dots, n$ so that $p_k=p$. Further, we write $P\subset U$ for a subset $U\subset X$, if $p_k \in U$ for each $k=0,\dots, n$.  We write $P\subset Q$ if the sequence of points in $P$ forms a subsequence of the points in $Q$ without gaps: that is $P=(p_0,\dots, p_m)$ and $Q=(q_0,\dots, q_n)$, $m\leq n$ and there exists a non-negative $s\leq n-m$ so that $q_{s+k}=p_k$ for each $k=0,\dots, m$. The path $P$ is called a \emph{sub-path} of $Q$.

For $\delta>0$, and $A\subset X$ a fixed closed set, we say that a discrete path $P=(p_0,\dots, p_n)$ is $(\delta,A,x)$-admissible, if $\mesh(P)\leq \delta$, $p_0\in A$  and $p_n=x$. The collection of all $(\delta,A,x)$ admissible discrete paths is denoted $\mathcal{P}(\delta,A,x)$.

Suppose that $f:X\to[0,M]$ is a bounded function for some $M>0$, $g:X\to [0,\infty)$ is a continuous bounded function, and $A\subset X$ is a closed subset. Then, for $\delta>0$ we define an \emph{approximating} function $\tilde{f}$ with data $(f,g,A,M,\delta)$ as
\[
\tilde{f}(x)=\min\{M,\inf_{(p_0,\dots, p_n) \in \mathcal{P}(\delta,A,x)} f(p_0) + \sum_{k=0}^{n-1} g(p_k)d(p_k,p_{k+1})\}.
\]

\begin{lemma}\label{lem:approx}
With the definitions and assumptions of this subsection, we have the following properties.
\begin{enumerate}
\item[A)] $\tilde{f}:X\to [0,M]$.
\item[B)] For each $x\in A$, $0\leq \tilde{f}(x)\leq f(x)$.
\item[C)] If $x\in A$ and $f(x)=0$, then $\tilde{f}(x)=0$.
\item[D)] For each $x,y\in X$ with $d(x,y)\leq \delta$, we have 
\begin{equation} \label{eq:lipa}
|\tilde{f}(x)-\tilde{f}(y)|\leq \max\{g(x),g(y)\}d(x,y).
\end{equation}  
\item[E)] $\lip_a[\tilde{f}](x)\leq g(x)$ for every $x\in X$.
\item[F)] $\tilde{f}$ is $\max\{M\delta^{-1}, \sup_{x\in X} g(x)\}$-Lipschitz.
\end{enumerate}
\end{lemma}
\begin{proof}

\begin{enumerate}
\item[A)] First, since each term in the infimum is non-negative, $\tilde{f}$ is also non-negative. Further, $\tilde{f}(x)\leq M$ for each $x\in X$ follows immediately from the definition.

\item[B)] Let $x\in A$. By noting that the discrete path $P=(x,x)$ is $(\delta,A,x)$-admissible,  we get $\tilde{f}(x)\leq f(x)+g(x)d(x,x)=f(x)$.

\item[C)] Let $x\in A$. By B) and A) we get $0\leq \tilde{f}(x)\leq f(x)=0$.

\item[D)] Let $x,y\in X$ be arbitrary with $d(x,y) \leq \delta$. Consider an arbitrary discrete $(\delta, A, x)$-admissible path $P=(p_0,\dots, p_n)$. We form a $(\delta, A,y)$-admissible path $Q=(q_0,\dots, q_{n+1})$ by adjoining $q_{n+1}=y$ and setting $q_i=p_i$ for $i\in \{0,\dots, n\}$. With such choices  
\[
f(q_0)+\sum_{k=0}^{n} g(q_k)d(q_k,q_{k+1}) \leq f(p_0)+\sum_{k=0}^{n-1} g(p_k)d(p_k,p_{k+1}) + g(x)d(x,y)
\]

Taking an infimum over all $(n,A,x)$-admissible paths and a minimum with $M$ on both sides of the inequality, we get $\tilde{f}(y)\leq \tilde{f}(x)+g(x)d(x,y)$. By switching the roles of $x$ and $y$, we obtain the desired inequality \eqref{eq:lipa}.

\item[E)]Take points $a,b\in X$. Then, by sending $a,b\to x$ and applying \eqref{eq:lipa} together with the continuity of $g$ to such a pair, we get
\[
\lip_a[\tilde{f}](x)\leq g(x).
\]

\item[F)] Let $L=\max\{M\delta^{-1}, \sup_{x\in X} g(x)\}$. If $x,y\in X$, with $d(x,y)\leq \delta$, then from B) we get $|\tilde{f}(x)-\tilde{f}(y)|\leq L d(x,y)$. On the other hand, if $x,y\in X$ with $d(x,y)>\delta$, then by A) we have $|\tilde{f}(x)-\tilde{f}(y)|\leq M\leq M\delta^{-1} d(x,y)\leq L d(x,y)$.
\end{enumerate} 
\end{proof}

\subsection{A compactness result for discrete curves}

The convergence of our approximating functions to $f$ is obtained via a contradiction argument, where we obtain a sequence of discrete paths $P^i$, and extract a subsequence converging to a curve $\gamma:[0,1]\to X$. In this subsection we give the definition of convergence that we use, and the main results for it.

We will need a piece of notation. For a subset $A \subset X$ and $x\in A$,  the distance from $a$ to $A$ is given by  $d(x,A)\defeq \inf_{a \in A} d(a,x)$.

First, to define our notion of convergence we fix an isometric Kuratowski embedding $\iota: X\to \ell_\infty(\N)$, and identify $X$ with its image under this embedding. \emph{A priori} the notions that follow may depend on the choice of such an embedding. For our ultimate argument, such a dependence will play no role.

If $P=(p_0,\dots, p_n)$ is a discrete path, we define its \emph{linearly interpolating curve} as follows. If $\len(P)=0$, we define $\gamma_P:[0,1]\to \ell_\infty(\N)$ by $\gamma(t)=p_0$ for each $t\in [0,1]$. If $\len(P)>0$ we define the \emph{sequence of interpolating times} $T_P=(t_0,\dots, t_n)$ by $t_0=0$ and $t_k = \sum_{i=0}^{k-1} d(p_i,p_{i+1})/\len(P)$ for $k=1,\dots, n$. Then, we define $\gamma_P:[0,1]\to \ell_\infty(\N)$ piecewise by linear interpolation in $\ell_\infty(\N)$: when $t\in [t_k,t_{k+1}]$ and $t_k=t_{k+1}$, we set $\gamma_P(t_k)=p_k$, and when $t\in [t_{k},t_{k+1}]$ and $t_{k+1}>t_k$, we set $\gamma(t)=((t-t_k)p_k + (t_{k+1}-t)p_{k+1})(t_{k+1}-t_k)^{-1}$. The following lemma is elementary to verify. 

\begin{lemma}\label{lem:basic} 
If $P$ is a discrete path, and $\gamma_P$ is its linearly interpolating curve, then $\gamma_P$ is $\len(P)$-Lipschitz, parametrized by unit speed, $\len(P)=\len(\gamma_P)$ and for each $t\in [0,1]$ there exists a $p\in P$ so that $d(\gamma_P(t),p)\leq \mesh(P)$.
\end{lemma}

We say that a sequence of discrete paths $P^i$ \emph{converges} to a curve $\gamma:[0,1]\to X$, 
if $\gamma_{P^i}$ converges uniformly to $\gamma$ and if $\lim_{i\to\infty}\mesh(P^i)=0$. While we do not need this, we note that this notion of convergence does not depend on the embedding to $\ell_\infty(\N)$ and could be defined intrinsically. Indeed, we could also define paths with jumps $\gamma'_P:[0,1]\to X$ piecewise: if $t\in [t_k,t_{k+1})$ we set $\gamma_P(t_k)=p_k$. Then, we say that $P^i$ converges to $\gamma$ if $\gamma'_{P^i}$ converges unifromy to $\gamma.$ It is straightforward to show that these two definitions are equivalent. However, it will reduce the number of technical issues to use the linearly interpolating curves. 

We need compactness results for discrete paths. First, we give first a simpler compactness statement, to illustrate the main idea in the setting where $X$ is compact.

\begin{lemma}\label{lem:easiercompact} If $X$ is a compact metric space, and if $\{P^i\}_{i\in\N}$ is a sequence of discrete curves with $\lim_{i\to\infty} \mesh(P^i)=0$ and $\sup_{i\in \N}\len(P^i)<\infty,$ then there exists a subsequence $i_k$ so that $P^{i_k}$ converges to a curve $\gamma:[0,1]\to X$.
\end{lemma}
\begin{proof}
Let $\gamma_{P^i}:[0,1]\to \ell_\infty(\N)$ be the linearly interpolating curve and $L=\sup_{i\in \N} \len(P^i)<\infty$. To prove the claim, we first need to show that the sequence of linearly interpolating curves $\gamma_{P^i}$ satisfies the assumptions of Lemma \ref{lem:arzela-ascoli}. First, Lemma \ref{lem:basic} gives that $\gamma_{P^i}$ is $L$-Lipschitz. Second, let $t\in [0,1]$, and consider $A_t=\{\gamma_{P^i}(t) : k\in \N\}.$ We will show that $A_t$ is precompact in $\ell_\infty(\N)$ by showing that $A_t$ is totally bounded.

Fix $\eta>0$, and $N\in \N$ so that for $i\geq N$ we have $\mesh(P^i)\leq \eta$. Let $K=X\cup\{\gamma_{P^1}(t), \dots, \gamma_{P^N}(t)\}$. We claim that $d(a,K)\leq \eta$ for each $a=\gamma_{P^i}(t)\in A_t$. For each $i\leq N$ this claim is trivial. For each $i> N$, by Lemma \ref{lem:basic}, there exists a $p^i \in P^i$ so that $d(\gamma_{P^i}(t),p^i)\leq \mesh(P^i)\leq \eta$. Since $p^i \in X$, we have $d(\gamma_{P^i}(t),X)\leq \eta$. 

The set $K$ is compact, and thus totally bounded. Thus, we can points $x_1,\dots, x_M \in K$ so that $K\subset \bigcup_{j=1}^M B(x_j,\eta)$. Combining with the previous paragraph, we get $A_t\subset \bigcup_{j=1}^M B(x_j,2\eta)$, which proves the totally boundedness.

Now, by Lemma \ref{lem:basic}, we have that a subsequence of $\gamma_{P^i}$ converges to some curve $\gamma:[0,1]\to \ell_\infty(\N)$. By Lemma \ref{lem:basic}, $d(\gamma(t),X)\leq \limsup_{i\to\infty} d(\gamma_{P^i}(t),X)\leq \lim_{i\to\infty} \mesh(P^i) = 0$, and thus the image of  $\gamma$ is contained in $X$. 
\end{proof}

If we were to work only in proper metric spaces, the previous Lemma would be quite sufficient for the proof of Theorem \ref{thm:densinenergy}. However, for non-proper spaces, we need to force the discrete curves to ``not pass far'' from a sequence of compact sets.

First, for a sequence of non-empty compact sets $K_n\subset X$, with $K_n\subset K_m$ for $n\leq m$, we call a sequence of continuous bounded functions $h_n:\ell_\infty(\N)\to [0,\infty]$ defined by
\begin{equation}\label{eq:defhn}
h_n(x)\defeq\sum_{k=1}^n \min\{nd(x,K_k),1\}.
\end{equation}
as \emph{a good sequence of functions} for $\{K_k\}_{k\in\N}$.

These functions penalize paths that travel far away from the compact sets $K_n$. Indeed, if $d(x,K_n)\geq \eta$, then $h_n(x)\geq n\min\{n\eta,1\}.$ If we \emph{assume} that a sum involving $h_n$ over a discrete path $P$ is controlled, then we can use this bound effectively to conclude that $P$ is contained within an $\eta$-neighborhood of $K_n$.

Since the proof of the following Lemma is nearly identical to the previous Lemma, we will slightly abbreviate the proof, and utilize the same notation.

\begin{lemma}\label{lem:compactness} Let $K_n\subset X$ be a sequence of compact sets and let  $M,L,\Delta>0$ be constants. Let further $h_n$ be a good sequence of functions for $\{K_n\}_{n\in \N}$. If $\{P^i=(p_0^i, \dots, p_{n(i)}^i)\}_{i\in \N}$ is a sequence of discrete paths, with $\lim_{i\to \infty} \mesh(P^i)=0$, $\len(P^i)\leq L$, $\diam(P^i)\geq \Delta$ and 
\[
\sum_{k=0}^{n(i)-1} h_i(p_k^i)d(p_k^i,p_{k+1}^i)\leq M
\]
for each $i\in \N$, then there exists a subsequence of $P^i$ which converges to a curve $\gamma:[0,1]\to X$.
\end{lemma}

\begin{proof}
The proof proceeds in the same way as Lemma \ref{lem:easiercompact}. First, for every $i\in \N$ the curve $\gamma_{P^i}$ is $L-$Lipschitz. Second, if a subsequence of $\gamma_{P^i}$ converges to some curve $\gamma:[0,1]\to \ell_\infty(\N)$, then by Lemma \ref{lem:basic}, we have $d(\gamma(t),X)\leq \limsup_{i\to\infty} d(\gamma_{P^i}(t),X)\leq \lim_{i\to\infty} \mesh(P^i) = 0$ for each $t\in [0,1]$, and thus the image of  $\gamma$ is contained in $X$.

Thus, the only aspect to verify, is the second assumption of Lemma \ref{lem:arzela-ascoli}, i.e. that the set $A_t=\{\gamma_{P^i}(t) : k\in \N\}$ is precompact in $\ell_\infty(\N)$ for each $t\in [0,1]$. Equivalently, that it is totally bounded, which is equivalent to showing that for each $\eta>0$, we can find a compact set $K$ so that $d(a,K)\leq \eta$ for each $a\in A_t$.

Fix $\eta>0$ and $t\in [0,1]$. Fix $N\in \N$ so that $\mesh(P^i)\leq \eta/8$ for each $i\geq N$. Then, fix $T = 8\lceil \eta^{-1} \rceil +\lceil 2M\min\{\eta/8, \Delta\}^{-1}\rceil$, and set $\tilde{N}=\max\{N,T\}$. Set $K = K_{\tilde{N}} \cup \{\gamma_{P^i}(t): i=1,\dots, \tilde{N}\}$.

Take any $i\in \N$ and consider $\gamma_{P^i}(t)\in A_t$. If $i\leq \tilde{N}$, we have $d(\gamma_{P^i}(t),K)=0\leq \eta$. Next, assume $i\geq \tilde{N}$. By Lemma \ref{lem:basic}, there exists a $p^i\in P^i$ so that $d(p^i, \gamma_{P^i}(t))\leq \mesh(P^i)\leq \eta / 8$. If $d(p^i, K_{\tilde{N}})\leq \eta/2$ for each $i\geq \tilde{N}$, then $d(\gamma_{P^i}(t),K)\leq \eta$, and we are done. Suppose therefore that we have some index $i\geq \tilde{N}$, with $d(p^i, K_{\tilde{N}})\geq \eta/2$. We will show that this leads to a contradiction completing the proof.

Let $Q^i=(q_0^i, \dots, q_{m(i)}^i) \subset P^i$ be the largest sub-path of $P^i$ which contains $p^i$ and which is contained in $B(p^i, \eta/4)$. Since $\mesh(P^i)\leq \eta/8$, and $\diam(P^i)\leq \Delta$, we have $\diam(Q^i)\geq \min\{\eta/8, \Delta\}$. Then, $d(q,K_{\tilde{N}})\geq \eta/4\geq \tilde{N}^{-1}$ for each $q\in Q^i$. Thus, $d(q, K_j)\geq \tilde{N}^{-1}$ for each $q\in Q^i$ and $j\leq \tilde{N}.$ In particular $h_i(q)\geq h_{\tilde{N}}(q)\geq \tilde{N} \geq \lceil 2M \min\{\eta/8, \Delta\}^{-1}\rceil$. Thus,

\[\sum_{k=0}^{m(i)-1} h_{i}(q_k^i)d(q_k^i,q_{k+1}^i)\geq  \lceil 2M\min\{\eta/8, \Delta\}^{-1}\rceil \len(Q^i) \geq 2M.
\]

Since $Q\subset P$, we have
\[2M \leq \sum_{k=0}^{m(i)-1} h_{i}(q_k^i)d(q_k^i,q_{k+1}^i)\leq \sum_{k=0}^{n(i)-1} h_i(p_k^i)d(p_k^i,p_{k+1}^i) \leq M,
\]
which is a contradiction.
\end{proof}

We will also need a slightly technical lower semicontinuity statement reminiscent of \cite[Proposition 4]{keith03}. The sums that appear in the statement, and in Lemma \ref{lem:compactness}, should be thought of as discrete Riemann sums.

\begin{lemma}\label{lem:lowersemicont}
Let $g:X\to [0,\infty]$ be a lower semicontinuous function, and assume that $\{g_i:X\to[0,\infty]\}_{i\in \N}$ is an sequence of continuous functions which converges to $g$ pointwise, with $g_i(x)\leq g_j(x)$ for each $x\in X$, and each $i\leq j$.  If $\{P^i=(p_0^i, \dots, p_{n(i)}^i)\}_{i\in\N}$ is a sequence of discrete paths, with $\sup_{i\in \N} \len(P^i)<\infty$ and which converges to a curve $\gamma:[0,1]\to X$, then 
\[
\int_\gamma g~ds\leq \liminf_{i\to\infty}\sum_{k=0}^{n(i)-1} g_i(p_k^i)d(p_k^i, p_{k+1}^i).
\]
\end{lemma}

\begin{proof} Let $\gamma_{P^i}$ be the linearly interpolating curves to $P^i$ and let $L=\sup_{i\in \N} \len(P^i)$.  Use the Tietze extension theorem to extend each $g_i:X \to \R$ to be a continuous function on $\ell_\infty(\N)$. By constructing the extensions recursively and taking maxima, we can ensure $g_i \leq g_j$ for $i\leq j$. Extend the function $g$ to a lower semicontinuous function on $\ell_\infty(\N)$ to satisfy  $g(x)=\lim_{i\to\infty} g_i(x)$, for all $x\in \ell_\infty(\N)$.  Denote the extensions still by the same letter. 

Fix for the moment an index $i\in \N$ and consider the function $g_i$. Since $\gamma_{P^j}$ converges uniformly to $\gamma$, when $j\to\infty$, we have that the set $K\subset \ell_\infty(\N)$ formed by the union of the images of the curves $\gamma_{P^j}, \gamma$ ($j\in \N$) is compact. On $K$, the function $g_i$ is uniformly continuous. 

Fix an $\epsilon>0$. Since $g_i|_K$ is uniformly continuous, there exists a $\delta>0$ so that if $x,y\in K$ and $d(x,y)\leq \delta$, then $|g_i(x)-g_i(y)|\leq \epsilon/L$. Choose then $N$ so large that $\mesh(P^j)\leq \delta$ for each $j\geq N$. If $T_{P^j}=(t^j_0, \dots, t^j_{n(j)})$ is the sequence of interpolating times for $P^j$, then $d(\gamma_{P^j}(t),p^j_k)\leq \mesh(P^j)\leq \delta$ for each $t\in [t^j_k, t^j_{k+1}]$ ($k=0,\dots, n(j)-1$). Consequently, for each $j\geq N$, we get
\begin{align*}
\left|\int_{\gamma_{P^j}} g_i(x) ~ds_x- \sum_{k=0}^{n(j)-1} g_i(p^{j}_k) d(p^{j}_k,p^{j}_{k+1})\right|&\leq \left|\sum_{k=0}^{n(j)-1} \int_{\gamma_{P^j}|_{[t_k^j,t_{k+1}^j]}} g_i(x) ~ds_x- \!\!\sum_{k=0}^{n(j)-1} g_i(p^{j}_k) d(p^{j}_k,p^{j}_{k+1})\right|\\
&=\left|\sum_{k=0}^{n(j)-1} \int_{\gamma_{P^j}|_{[t_k^j,t_{k+1}^j]}} (g_i(x)-g_i(p_j^k)) ~ds_x\right|\\
&\leq \len(P^j)\epsilon/L \leq \epsilon.
\end{align*}
On the last line we used the fact that $\len(P^j)=\len(\gamma_{P^j})$ by Lemma \ref{lem:basic}. Since $\epsilon>0$ was arbitrary, by sending $j\to\infty$ we get
\begin{align*}
\lim_{j\to\infty}& \left|\int_{\gamma_{P^j}} g_i(x) ~ds- \sum_{k=0}^{n(j)-1} g_i(p^{j}_k) d(p^{j}_k,p^{j}_{k+1})\right|=0.
\end{align*}

Combine this with the lower semi-continuity of curve integrals (see e.g.  the argument in \cite[Proposition 4]{keith03}) and the fact that $g_i\leq g_j$ for $i\leq j$, we have for each $i\in \N$

\begin{align*}
\int_\gamma g_i ~ds &\leq  \liminf_{j\to\infty} \int_{\gamma_{P^j}} g_i(x) ~ds =  \liminf_{j\to\infty}\sum_{k=0}^{n(j)-1} g_i(p_k^j)d(p_k^j, p_{k+1}^j) \\
&\leq  \liminf_{j\to\infty}\sum_{k=0}^{n(j)-1} g_j(p_k^i)d(p_k^j, p_{k+1}^j).
\end{align*}

Now, by letting $i\to\infty$ and by using monotone convergence on the left hand side, we obtain the statement of the lemma.

\end{proof}

\subsection{Proof of Theorem \ref{thm:densinenergy}}

With the stage set, we are now able to prove the main theorem of the paper. The proof will proceed by first giving some reductions, stating a goal, and then constructing a sequence of approximations that reaches this goal. The tricky bit of the proof is showing a pointwise convergence result for the approximating sequence, which is done using the Lemmas above and a contradiction argument.

\begin{proof}[Proof of Theorem \ref{thm:densinenergy}] Let $f\in N^{1,p}(X)$ and $g_f(x)  \in L^p(X)$ be its minimal $p$-weak upper gradient. We proceed to find an approximation in a series of steps. First, we reduce the claim to non-negative functions, and show that bounded functions are dense in  $N^{1,p}(X)$. Then we show that boundedly supported functions are dense among these bounded functions. Finally, we show that any non-negative function with bounded support can be approximated by a Lipschitz function in energy. If we wish to find the approximating sequence of Lipschitz functions for $f$, we would employ these three approximation schemes and use a diagonal argument.
\vskip.3cm
\noindent \textbf{Reduction to non-negative:} We can write $f=f_+ - f_-,$ where $f_+=\max\{f,0\}$ and $f_- = \max\{-f,0\}$. Then $g_f = g_{f_+}+ g_{f_-}$, by \cite[Proposition 6.3.22]{shabook}\footnote{This proposition states that if $u,v\in N^{1,p}(X)$ and $u=v$ on a set $A \subset X$, then for $\mu$-almost every $x \in A$ we have $g_u(x)=g_v(x)$ for their minimal $p$-weak upper gradients.} and if we approximate $f_\pm$ in energy by a sequence $f_{\pm}^n$, then we get that $f_{+}^n-f_{-}^n$, by Lemma \ref{lem:convergence}, approximates $f$ in energy. In what follows, we assume that $f$ is non-negative.
\vskip.3cm
\noindent \textbf{Reduction to a bounded case:} Consider first functions $f_M =  \min\{f,M\}$ for all $M>0$. Then, we have $f_M \to_{M\to\infty} f$ in $L^p(X)$, and $g_M=g_f|_{X \setminus f^{-1}[0,M]}$ is a minimal $p$-weak upper gradient for $f-f_M$ \cite[Proposition 6.3.22]{shabook}.  Indeed, then $g_M \to 0$ in $L^p$. Thus, $f_M$ converges to $f$ in norm in $N^{1,p}(X)$, and we have that bounded functions are dense. Assume thus in what follows, that there is some $M$ so that $0\leq f\leq M$.

\vskip.3cm
\noindent \textbf{Reduction to bounded support:}
Similarly, if $x_0 \in X$ is any fixed point, and we consider the sequence $f_R(x) = f \psi_R(x)$, where $\psi_R(x) = \max\{0,\min\{1, R-d(x_0,x)\}\}$ and $R\in \N$, $R>0$. The sequence $0\leq \psi_R \leq 1$ is $1$-Lipschitz and $f_R\to f$ pointwise and in $L^p(X)$. Further, the function $f-f_R$ has a weak upper gradient $g_R=1_{X \setminus B(x_0,R-1)}g_f + f 1_{B(x_0,R+1)\setminus B(x_0,R)}$ as follows from the Leibniz rule in \cite[Proposition 6.3.28 and Proposition 6.3.22]{shabook}. Thus $g_R \to 0$ and we get that $f_R\to f$ in $N^{1,p}(X)$. Therefore functions with bounded support are dense, and we may assume that there is an $R\geq 4$ and $x_0 \in X$ with $f(x) = 0$ for $x\in X \setminus B(x_0,R)$.  

\vskip.3cm
\noindent \textbf{Goals of proof:} Let $\epsilon \in (0,1)$ be fixed. We  will show that we can find functions $g_\epsilon \in L^p(X)$ and $f_\epsilon \in \LIP_b(X) \cap N^{1,p}(X)$ so that $g_\epsilon \geq \lip_a[f_\epsilon]$,
\begin{equation}\label{eq:gepsilon}
\int_X |g_\epsilon-g_f|^p d\mu \leq \epsilon,
\end{equation}
and
\begin{equation}\label{eq:fcondition}
\int_X |f-f_\epsilon|^p d\mu \leq \epsilon. 
\end{equation}
The claim of the Theorem then follows directly from Lemma \ref{lem:convergence}, together with $\lip_a[f_\epsilon] \geq g_{f_\epsilon}$, by choosing a sequence of $\epsilon$ with $\epsilon \searrow 0$.

\vskip.3cm 
\noindent \textbf{Choice of $g_\epsilon$ and $g_n$:} First,  by Remark \ref{rmk:vitalicaratheodory} we can choose a lower semicontinuous function $g_1(x) \geq g_f$ with the property that $g_1$ is a lower semicontinuous upper gradient for $f$ and so that 

\begin{equation}\label{eq:g1}
\int_X |g_1-g_f|^p d\mu \leq \epsilon 4^{-2p}.
\end{equation}

Further, since $0$ is an upper gradient for $f$ when restricted to the set $X\setminus B(x_0,R)$, we can set $g_1(x)=g_f(x)=0$ for all $x \in X \setminus B(x_0, 4R)$. This follows directly from the definition of an upper gradient, and cutting the curve appearing in the defining inequality \eqref{eq:ug} to pieces contained within the ball $B(x_0,2R)$, and its complement.

By Lusin's theorem and inner regularity of $\mu$, we can choose an increasing sequence of compact sets $K_n \subset B(x_0,2R)$ so that $\mu(B(x_0,2R) \setminus K_n) \leq \epsilon n^{-p}4^{-n-2p}$ and so that $f|_{K_n}$ is continuous. Choose a $\sigma \in (0,1)$ so that $\mu(B(x_0,4R))\sigma^p \leq \epsilon 4^{-2p}$, and define again $\psi_{2R}(x) = \max\{0,\min\{1, 2R-d(x_0,x)\}\}$

Define 
\begin{equation}\label{eq:gepsdef}
g_\epsilon(x) \defeq g_1(x) + \sigma \psi_{2R}(x)+\sum_{n=1}^\infty 1_{B(x_0,2R) \setminus K_n}(x). 
\end{equation}
Inequality \eqref{eq:gepsilon} follows from  \eqref{eq:gepsdef} and \eqref{eq:g1}. Also, $g_\epsilon$ is lower-semicontinuous. 
 
 Choose any increasing sequence of bounded continuous functions $\{\tilde{g}_n\}_{n\in\N}$ converging to $g_1$ with  $0\leq \tilde{g}_m \leq \tilde{g}_n \leq g_1$ for $m\leq n$ and define 
 \begin{equation}\label{def:gn}
 g_n(x) = \tilde{g}_n(x)  +\sigma \psi_{2R}(x)+ \sum_{k=1}^n \min\{nd(x,K_k),1_{B(x_0,2R)}(x)\} .
 \end{equation}
 Then, we get that $0 \leq g_n\leq g_\epsilon$ and $g_n$ converges to $g_\epsilon$ as $n \to \infty$.  Finally, choose $L$ so that $\mu(B(x_0,2R) \setminus K_L) \leq \epsilon (2M)^{-p}$, and define the closed set $A\defeq K_L \cup X\setminus B(x_0,R)$. Since $f|_{K_L}$ and $f|_{X\setminus B(x_0,R)}$ are continuous, then $f|_A$ is continuous.
 
 We remark, that in equations \eqref{eq:gepsdef} and \eqref{def:gn}, we could avoid adding the summation term to $g_\epsilon$ and $g_n$ if the space was proper. This is one place where properness would yield a simplification. In this case, we would use Lemma \ref{lem:easiercompact} (with a bounded subset $\overline{B(x_0,2R)}$ replacing $X$) instead of Lemma \ref{lem:compactness}.
 \vskip.3cm 
\noindent \textbf{Approximating sequence $f_n$:} Define for each $n\in \N$ the approximating function $f_n$ with data $(f,g_n,A,M,n^{-1})$ by the formula

\begin{equation}\label{eq:fndef}
f_n(x):=\min\left\{M,\inf_{p_0,\dots,p_N} f(p_0)+\sum_{k=0}^{N-1} g_n(p_k)d(p_k,p_{k+1})\right\},
\end{equation}
where the infimum is taken over all $(n^{-1},A,x)$ admissible discrete paths $(p_0,\dots, p_N)$. 

Lemma \ref{lem:approx} shows that that $f_n:X\to [0,M]$ is a Lipschitz functions with $\lip_a[f_n]\leq g_n\leq g_\epsilon$, and $f_n|_A\leq f|_A$. From this lemma, we also get that $f_n(x)=0$ for each $x\not\in X\setminus B(x_0,R)$. Thus $f_n \in \LIP_b(X)\cap N^{1,p}(X)$.

The main step remaining is to show that for each $x\in K_L$ we have $\lim_{n\to\infty}f_n(x) = f(x)$ -- that is, we have pointwise convergence. Indeed, suppose that we have shown this. By Lebesgue dominated convergence and  $f_n(x)=f(x)=0$ on $A\setminus K_L$, we get for $n$ large enough that 
\[
\int_{A} |f_n-f|^p ~d\mu = \int_{K_L} |f_n-f|^p ~d\mu \leq \epsilon/2.\]
 We have $f_n(x)=f(x)=0$ for $x\in X\setminus B(x_0,R)$ and $f_n(x),f(x) \in [0,M]$ for $x\in B(x_0,R)\setminus K_L$. Thus, by the choice of $L$, we get

\[
\int_X |f-f_n|^p ~d\mu = \int_{K_L} |f-f_n|^p ~d\mu + \int_{B(x_0,R)\setminus K_L} |f-f_n|^p ~d\mu \leq \epsilon/2 + M^p \mu(B(x_0,R)\setminus K_L) \leq \epsilon.
\]

Thus,  the functions $f_\epsilon=f_n$ of $n$ large enough and $g_\epsilon$  realize all the aspects of the goal of the proof. We are left to show that $f_n|_{K_L}$ converges to $f|_{K_L}$ pointwise.\footnote{In fact, it follows from Dini's theorem, since $f_n\leq f_m$ for $n\leq m$, that $f_n|_{K_L}$ converges uniformly to $f|_{K_L}$. We do not need this claim here. }

\vskip.3cm \noindent \textbf{Pointwise convergence $\lim_{n\to\infty} f_n(x) = f(x)$ for $x\in K_L$:} Assume that there exists some  $x\in K_L$ for which pointwise convergence fails.  

For $n\geq m$ we have $g_n\leq g_m$ and thus $f_n\leq f_m$, since for each $x\in X$ any $(m^{-1},A,x)$-admissible path is also $(n^{-1},A,x)$-admissible. Thus, the sequence $\{f_n(x)\}_{n\in\N}$ is increasing in $n\in \N$ and by Lemma \ref{lem:approx} we have $f_n(x) \leq f(x)$. In particular, the limit $\lim_{n\to\infty} f_n(x)$ exists. Since the limit is not equal to $f(x)$,  there exists a constant $\delta>0$, so that $\lim_{n\to\infty}f_n(x) \leq f(x)-\delta$. Since we have $f_n(x)\geq 0$ for each $n\in \N$, we also must have $f(x)\geq \delta$, and so $x\in B(x_0,R)$.

Further, since $f(x)\leq M$, we get $f_n(x) \leq M-\delta$ for each $n\in\N$.  By the definition  of $f_n(x)$, we obtain discrete paths $P^n=(p^n_{0},\dots, p^{n}_{N_n})$ which are $(n^{-1},A,x)$-admissible, and with 

\begin{equation}\label{eq:boundn}
f(p_0^n) + \sum_{k=0}^{N_n-1} g_n(p^n_k) d(p^n_k,p^n_{k+1}) < f(x)-\delta/2 \leq M.
\end{equation}
 Our contradiction will be obtained by finding a curve in the limit of these discrete paths.

First, we restrict to a sub-path. Let $Q^n=(q^n_0, \dots,q^n_{M_n})$ be the largest sub-path of $P^n$ with $q^n_{M_n}=x$ and with $Q^n \subset B(x_0,3R/2)$. If $Q^n=P^n$, then clearly $Q^n$ is $(n^{-1},A,x)$-admissible and $f(q_0^n)=f(p_0^n)$. If $Q^n \subsetneq P^n$, then since $\mesh(P^n)\leq 1 < R/2$ and since $Q^n$ is the largest sub-path, we must have $q^n_0 \in B(x_0,3R/2) \setminus B(x_0,R)$.  Thus, $q^n_0\in A$, and $Q^n$ is still $(n^{-1},A,x)$-admissible. Also, $0=f(q_0^n)\leq f(p_0^n)$. These give that

\begin{equation}\label{eq:boundq}
f(q_0^n) + \sum_{k=0}^{M_n-1} g_n(q^n_k) d(q^n_k,q^n_{k+1}) \leq f(p_0^n) + \sum_{k=0}^{N_n-1} g_n(p^n_k) d(p^n_k,p^n_{k+1}) < f(x)-\delta/2 \leq M.
\end{equation}

We will show that $\{Q^n\}_{n\in \N}$ satisfies the assumptions of Lemma \ref{lem:compactness}.  This consists of verifying the mesh, lengh, diameter and $h$-sum conditions. Here, we would only need to verify the mesh-, and length- conditions and use Lemma \ref{lem:easiercompact} on the subset $\overline{B(x_0,2R)}$, if $X$ was proper. Fix $n\in \N$.

\begin{enumerate}
\item \textbf{Mesh:} We have $\lim_{n\to\infty}\mesh(Q
^n)\leq \lim_{n\to\infty} n^{-1}=0$, since $Q^n$ is $(n^{-1},A,x)$-admissible.
\item \textbf{Length:}  
From \eqref{def:gn} and $Q^n \subset B(x_0,3R/2)$ (and $R\geq 4$),  we get that $g_n(q)\geq \sigma$ for every $q\in Q^n$. Thus, by \eqref{eq:boundq} and by setting $L\defeq \sigma^{-1}M$ we get

\begin{equation}\label{eq:boundlen}
\len(Q^n) = \sigma^{-1}\sum_{k=0}^{M_n-1} \sigma d(q^n_k,q^n_{k+1}) \leq \sigma^{-1}\left(f(q_0^n) + \sum_{k=0}^{M_n-1} g_n(q^n_k) d(q^n_k,q^n_{k+1})\right) \leq L.
\end{equation}
\item \textbf{Diameter:} 
We stated earlier in the proof that $f|_{A}$ is continuous. Thus, we can find a constant $\Delta>0$, so that $y\in A$ and $d(x,y) \leq \Delta$ implies that $|f(x)-f(y)|\leq \delta/4$.  We get $f(q_0^n)<f(x)-\delta/2$ from Equation \eqref{eq:boundq}. Since $q_0^n\in A$, we get $d(q_0^n,x)=d(q_0^n,q_{M_n}^{n})\geq \Delta$. In particular, $\diam(Q^n)\geq \Delta$.
\item \textbf{$h$-sum:} Let $h_n$ be the good sequence of functions for $\{K_n\}$, which were defined in Equation \eqref{eq:defhn}. Note that, by definition $h_n|_{B(x_0,2R)}\leq g_n|_{B(x_0,2R)}$, and thus again from Equation \eqref{eq:boundq}, 
\[
\sum_{k=0}^{M_n-1} h_n(q^n_k) d(q^n_k,q^n_{k+1}) \leq \sum_{k=0}^{M_n-1} g_n(q^n_k) d(q^n_k,q^n_{k+1}) \leq M.
\]
\end{enumerate}

As a consequence, Lemma \ref{lem:compactness} shows that a subsequence of $Q^n$ converges to a  curve $\gamma:[0,1]\to X$. By reindexing, we may pass to this subsequence.
  
Since $q_0^n\in A$, and $\lim_{n\to\infty} q_0^n =\gamma(0)$, we get that $\gamma(0)\in A$. Similarly, $\gamma(1)=x$. Recall that $f|_{A}$ is continuous. Therefore, $f(\gamma(0)) = \lim_{n\to\infty} f(q^{n}_0)$.  Then, by Lemma \ref{lem:lowersemicont} applied to $Q^n,\gamma,g_n$ and $g_\epsilon$ with $g_n$ converging to $g_\epsilon$, we get
 
 \[
 f(\gamma(0)) + \int_\gamma g_\epsilon ~ds  \leq \liminf_{n\to\infty} f(q^{n}_0) + \sum_{k=0}^{M_n-1} g_{n}(q^n_k) d(q^{n}_k,q^{n}_{k+1}) \leq f(x)-\delta/2.\]
 
 In the last step, we used Inequality \eqref{eq:boundq}.  We obtain
 
 \[
 f(\gamma(1))-f(\gamma(0)) > \int_\gamma g_\epsilon ~ds,
 \] 
 which contradicts $g_\epsilon$ being a true upper gradient.

\end{proof}

\begin{remark} The proof above shows a more technical statement, which we highlight for purposes of future work. Suppose that $f$ is non-negative, boundedly supported in $B(x_0,R)$ and bounded by $M$ and has a lower semi-continuous upper gradient $g_1$.  Then if we fix any $\sigma>0$, and any increasing sequence of compact sets $K_n$, and define $g_\epsilon$ as in \eqref{eq:gepsdef},  there is a sequence of Lipschitz functions $f_m$ with $\lip_a[f_m] \leq g_\epsilon$ and $f_m \to f$ pointwise on each compact set $K_L$, for $L$ fixed.  Further,  if $X$ is proper and $f$ is continuous, then we do not need to consider the exhaustion by compact sets $K_n$ and get convergence on all of $B(x_0,R)$.  Also, if $g_1$ is bounded below on $B(x_0,2R)$, we do not even need to add $\sigma \psi_{2R}$, which is only added to ensure the length bound for $P^n$. 

A natural follow up work would consider these techniques in other Banach function spaces and associated Newton-Sobolev spaces, such as authors have done in \cite{helisobolev, lorentzsob,hasto1,hasto2}. See also the versions in general Banach function spaces in \cite{malylukas}.  
In these other function spaces, one would need to first ensure a lower semi-continuous upper gradient $g_1$ which is close in norm to the minimal one (by a version of Vitali-Caratheodory as in Remark \ref{rmk:vitalicaratheodory} and  \cite{minimallukas}). Then, check an appropriate version of Lemma \ref{lem:convergence}.  Finally, one would need to argue that the choices of $K_n$, and $\sigma$ can be made so that $g_\epsilon$ and $g_1$ are close in norm -- which relies on some absolute continuity and monotone convergence in the applicable Banach function space. If one wishes, in proper metric spaces this should be slightly easier.  For this argument, some form of Vitali-Caratheodory theorem holding for the Banach function space seems necessary, see \cite{malylukas}.  Further ideas or techniques,  such as some form of differential structure,  would be needed to upgrade the density in energy to density in norm. For such ideas, see \cite{teriseb}.
\end{remark}

\bibliographystyle{acm}
\bibliography{pmodulus}
\def\cprime{$'$}

\end{document}